\documentclass[reqno,english]{amsart}
\usepackage[T1]{fontenc}
\usepackage[latin9]{inputenc}
\usepackage{babel}
\usepackage{refstyle}
\usepackage{float}
\usepackage{units}
\usepackage{mathrsfs}
\usepackage{mathtools}
\usepackage{enumitem}
\usepackage{amstext}
\usepackage{amsthm}
\usepackage{amssymb}
\usepackage{graphicx}
\usepackage[all]{xy}
\usepackage[unicode=true,pdfusetitle,
 bookmarks=true,bookmarksnumbered=false,bookmarksopen=false,
 breaklinks=false,pdfborder={0 0 1},backref=false,colorlinks=false]
 {hyperref}
\hypersetup{
 colorlinks=true,citecolor=blue,linkcolor=blue,linktocpage=true}

\makeatletter

\AtBeginDocument{\providecommand\defref[1]{\ref{def:#1}}}
\AtBeginDocument{\providecommand\thmref[1]{\ref{thm:#1}}}
\AtBeginDocument{\providecommand\corref[1]{\ref{cor:#1}}}
\AtBeginDocument{\providecommand\lemref[1]{\ref{lem:#1}}}
\AtBeginDocument{\providecommand\figref[1]{\ref{fig:#1}}}
\AtBeginDocument{\providecommand\exaref[1]{\ref{exa:#1}}}
\AtBeginDocument{\providecommand\secref[1]{\ref{sec:#1}}}
\RS@ifundefined{subsecref}
  {\newref{subsec}{name = \RSsectxt}}
  {}
\RS@ifundefined{thmref}
  {\def\RSthmtxt{theorem~}\newref{thm}{name = \RSthmtxt}}
  {}
\RS@ifundefined{lemref}
  {\def\RSlemtxt{lemma~}\newref{lem}{name = \RSlemtxt}}
  {}

\numberwithin{equation}{section}
\numberwithin{figure}{section}
\numberwithin{table}{section}
\theoremstyle{plain}
\newtheorem{thm}{\protect\theoremname}[section]
\theoremstyle{plain}
\newtheorem{lem}[thm]{\protect\lemmaname}
\theoremstyle{definition}
\newtheorem{defn}[thm]{\protect\definitionname}
\theoremstyle{remark}
\newtheorem{rem}[thm]{\protect\remarkname}
\theoremstyle{plain}
\newtheorem{cor}[thm]{\protect\corollaryname}
\theoremstyle{definition}
\newtheorem{example}[thm]{\protect\examplename}
\theoremstyle{remark}
\newtheorem*{acknowledgement*}{\protect\acknowledgementname}

\@ifundefined{date}{}{\date{}}
\allowdisplaybreaks
\usepackage{needspace}
\usepackage{refstyle}
\usepackage{enumitem}
\usepackage{bbm}

\newref{lem}{refcmd={Lemma \ref{#1}}}
\newref{thm}{refcmd={Theorem \ref{#1}}}
\newref{cor}{refcmd={Corollary \ref{#1}}}
\newref{sec}{refcmd={Section \ref{#1}}}
\newref{sub}{refcmd={Section \ref{#1}}}
\newref{subsec}{refcmd={Section \ref{#1}}}
\newref{chap}{refcmd={Chapter \ref{#1}}}
\newref{prop}{refcmd={Proposition \ref{#1}}}
\newref{exa}{refcmd={Example \ref{#1}}}
\newref{tab}{refcmd={Table \ref{#1}}}
\newref{rem}{refcmd={Remark \ref{#1}}}
\newref{def}{refcmd={Definition \ref{#1}}}
\newref{fig}{refcmd={Figure \ref{#1}}}

\setlist[enumerate]{itemsep=5pt,topsep=3pt}
\setlist[enumerate,1]{label=\textup{(}\roman*\textup{)},ref=\roman*}
\setlist[enumerate,2]{label=(\alph*),ref=\theenumi \alph*}

\usepackage{tikz}
\usetikzlibrary{decorations.pathreplacing}

\@ifundefined{showcaptionsetup}{}{%
 \PassOptionsToPackage{caption=false}{subfig}}
\usepackage{subfig}
\AtBeginDocument{
  
}

\makeatother

\providecommand{\acknowledgementname}{Acknowledgement}
\providecommand{\corollaryname}{Corollary}
\providecommand{\definitionname}{Definition}
\providecommand{\examplename}{Example}
\providecommand{\lemmaname}{Lemma}
\providecommand{\remarkname}{Remark}
\providecommand{\theoremname}{Theorem}

\begin{document}
\title[]{New boundaries for positive definite functions}
\author{Palle Jorgensen}
\address{(Palle E.T. Jorgensen) Department of Mathematics, The University of
Iowa, Iowa City, IA 52242-1419, U.S.A. }
\email{palle-jorgensen@uiowa.edu}
\urladdr{http://www.math.uiowa.edu/\textasciitilde jorgen/}
\author{Feng Tian}
\address{(James F. Tian) Mathematical Reviews, 416 4th Street Ann Arbor, MI
48103-4816, U.S.A.}
\email{jft@ams.org}
\begin{abstract}
With view to applications in stochastic analysis and geometry, we
introduce a new correspondence for positive definite kernels (p.d.)
$K$ and their associated reproducing kernel Hilbert spaces. With
this we establish two kinds of factorizations: (i) Probabilistic:
Starting with a positive definite kernel $K$ we analyze associated
Gaussian processes $V$. Properties of the Gaussian processes will
be derived from certain factorizations of $K$, arising as a covariance
kernel of $V$. (ii) Geometric analysis: We discuss families of measure
spaces arising as boundaries for $K$. Our results entail an analysis
of a partial order on families of p.d. kernels, a duality for operators
and frames, optimization, Karhunen--Lo\`eve expansions, and factorizations.
Applications include a new boundary analysis for the Drury-Arveson
kernel, and for certain fractals arising as iterated function systems;
and an identification of optimal feature spaces in machine learning
models.
\end{abstract}


\maketitle
\tableofcontents{}

\section{\label{sec:Intro}Introduction}

The notion of a \emph{positive definite} (p.d.) \emph{kernel} has
come to serve as a versatile tool in a host of problems in pure and
applied mathematics. The abstract notion of a p.d. kernel is in fact
a generalization of that of a positive definite function, or a positive-definite
matrix. Indeed, the matrix-point of view lends itself naturally to
the particular factorization question which we shall address below.
The general idea of p.d. kernels arose first in various special cases
in the first half of 20th century: It occurs in work by J. Mercer
in the context of solving integral operator equations; in the work
of G. Szeg\H{o} and S. Bergmann in the study of harmonic analysis
and the theory of complex domains; and in the work by N. Aronszajn
in boundary value problems for PDEs. It was Aronszajn who introduced
the natural notion of reproducing kernel Hilbert space (RKHS) which
will play a central role here. References covering the areas mentioned
above include \cite{MR3687240,MR0051437,MR562914,IM65,jorgensen2018harmonic,MR0277027,MR3882025},
and \cite{MR3721329}.

Let $X$ be a set and $K$ a function $X\times X\rightarrow\mathbb{C}$.
We say that $K$ is \emph{positive definite} (p.d.) iff (Def), for
all finite subsets $F_{N}=\left\{ x_{i}\right\} _{i=1}^{N}$, the
$N\times N$ matrix $\left(K\left(x_{i},x_{j}\right)\right)$ is positive
semidefinite. From the classical theory of Aronszajn et. al. (see
the papers cited above), there is then a canonical \emph{reproducing
kernel Hilbert space} (RKHS), denoted $\mathscr{H}\left(K\right)$,
with the following properties:
\begin{enumerate}
\item \label{enu:rk1}$K\left(\cdot,x\right)\in\mathscr{H}\left(K\right)$
for all $x\in X$;
\item \label{enu:rk2}$\mathscr{H}\left(K\right)$ is a Hilbert space of
functions on $X$, and , for all $x\in X$, we have 
\begin{equation}
f\left(x\right)=\left\langle K\left(\cdot,x\right),f\right\rangle _{\mathscr{H}\left(K\right)},\quad f\in\mathscr{H}\left(K\right).\label{eq:A1}
\end{equation}
Note that in (\ref{eq:A1}), the inner product $\left\langle \cdot,\cdot\right\rangle _{\mathscr{H}\left(K\right)}$
is linear in the second variable, i.e., in $f$.
\end{enumerate}
We shall need the following
\begin{lem}
\label{lem:A1}A function $f$ on $X$ is in $\mathscr{H}\left(K\right)$
if and only if there is a finite constant $C=C_{f}$ (depends only
on $f$) such that, for all $F_{N}=\left\{ x_{i}\right\} _{i=1}^{N}\subset X$,
finite subset, we have 
\begin{equation}
\left|\sum_{i=1}^{N}c_{i}f\left(x_{i}\right)\right|^{2}\leq C_{f}\sum_{i=1}^{N}\sum_{j=1}^{N}\overline{c}_{i}c_{j}K\left(x_{i},x_{j}\right).\label{eq:A2}
\end{equation}
Moreover, $\left\Vert f\right\Vert _{\mathscr{H}\left(K\right)}^{2}$
is then the infimum of the constant $C_{f}$ which occurs in (\ref{eq:A2}).
\end{lem}

\begin{proof}
We refer to the literature; see especially \cite{MR1986785,MR0008639,MR3860446}.
\end{proof}
We shall study sigma-finite measure spaces $\left(B,\mathscr{F}_{B},\mu\right)$
and functions 
\begin{equation}
X\ni x\longmapsto k_{x}^{\left(B\right)}\in L^{2}\left(B,\mu\right)\label{eq:A3}
\end{equation}
such that the associated p.d. kernel $K^{\left(\mu\right)}$ on $X\times X$
satisfying:
\[
K^{\left(\mu\right)}\left(x,y\right)=\left\langle k_{x}^{\left(B\right)},k_{y}^{\left(B\right)}\right\rangle _{L^{2}\left(B,\mu\right)}
\]
is comparable with $K$ itself (in a sense which we make precise in
\defref{FS} below).

Right up to the present, p.d. kernels have arisen as powerful tools
in many and diverse areas of mathematics. A partial list includes
the areas listed above in the Introduction. An important new area
of application of RKHS theory includes the following \cite{MR1120274,MR1200633,MR1473250,MR1821907,MR1873434,MR1986785,MR2223568,MR2373103}.

\section{Partial Order on Kernels, Operators, and Factorizations}

In this section we discuss two kinds of factorizations: (i) Probabilistic:
Starting with a positive definite kernel $K$ we present a canonical
Gaussian process $V$ which yields a factorization of $K$. In particular,
$K$ arises as covariance kernel for $V$; see (\ref{eq:g4}). And
(ii) Geometric: we discuss families of measure spaces which arise
as boundaries for $K$; see Definitions \ref{def:FS} and \ref{def:D5}.

\subsection{Kernels in Probability and in Geometry}

We want to stress below the distinction between the notions of factorization
introduced in eq. (\ref{eq:k4}) and (\ref{eq:k5}) from \defref{FS}.

In summary, condition (\ref{eq:k4}) has its roots in probability
and in the study of Gaussian processes. For the following result we
cite the papers \cite{AD92,MR2966130,MR3402823,MR3687240,MR2373103,MR0088706,MR3379106,MR0301806,MR562914,MR1176778,IM65,MR2053326,MR3701541,doi:10.1002/9781119414421.ch2,MR3574374,Jorgensen2017,zbMATH06690858,MR3888850,MR4020693,MR1472736,MR3526117,PaSc75,MR1478165,MR3571702,MR3688637},
and especially \cite{MR735967}.
\begin{thm}
Let $X\times X\xrightarrow{\;K\;}\mathbb{C}$ be a fixed positive
definite (p.d.) kernel. Then there is a Gaussian process 
\begin{equation}
\left(V_{x}\right)_{x\in X}\subset L^{2}\left(\Omega,\mathscr{F},\mathbb{P}\right)\label{eq:g1}
\end{equation}
realized on a probability space $\left(\Omega,\mathscr{F},\mathbb{P}\right)$
as follows:

For every $F_{n}:=\left\{ x_{i}\right\} _{i=1}^{n}$ finite subset
of $X$, let $\mathscr{F}_{n}$ be the corresponding cylinder sigma-algebra
of subsets of $\mathbb{C}^{X}$ (= the corresponding infinite Cartesian
product) with measures specified as follows: If $A_{n}\in\mathscr{F}_{n}$,
set 
\begin{equation}
\mathbb{P}_{n}\left(A_{n}\right)=\int_{\mathbb{C}^{n}}\chi_{A_{n}}\left(z\right)dg_{n}\left(z\right)\label{eq:g2}
\end{equation}
where $g_{n}\left(\cdot\right)$ is the Gaussian density on $\mathbb{C}^{n}$
with mean zero, and $n\times n$ covariance matrix $\left(K\left(x_{i},x_{j}\right)\right)_{\left(i,j\right)}$.
A standard application of Kolmogorov's consistency theorem then yields
a measure $\mathbb{P}$ and $\Omega=\mathbb{C}^{X}$, with $\mathscr{F}=$
the sigma algebra of subsets generalized by the cylinder subsets such
that 
\begin{equation}
\mathbb{E}_{\mathbb{P}}\left(\psi\chi_{A_{n}}\right)=\mathbb{E}_{n}\left(\mathbb{E}\left(\psi\mid\mathscr{F}_{n}\right)\chi_{A_{n}}\right)\label{eq:g3}
\end{equation}
holds for all $\psi\in L^{2}\left(\Omega,\mathscr{F},\mathbb{P}\right)$,
all $n$, and all $A_{n}\in\mathscr{F}_{n}$. Then for $\omega\in\Omega\left(=\mathbb{C}^{X}\right)$
and $x\in X$, set $V_{x}\left(\omega\right)=\omega\left(x\right)$.
It follows that $\left\{ V_{x}\right\} _{x\in X}$ is the desired
Gaussian process; in particular, 
\begin{equation}
\mathbb{E}\left(\overline{V}_{x}V_{y}\right)=K\left(x,y\right)\label{eq:g4}
\end{equation}
for all $\left(x,y\right)\in X\times X$.

Moreover, the assignment 
\begin{equation}
\underset{\left(\text{the RKHS}\right)}{\mathscr{H}\left(K\right)}\ni K\left(\cdot,x\right)\xrightarrow{\;T_{K}\;}V_{x}\in L^{2}\left(\Omega,\mathscr{F},\mathbb{P}\right)\label{eq:g5}
\end{equation}
defines a (canonical) isometry (extension by linearity and norm-closure).
\end{thm}

We say that (\ref{eq:g4}) is the \emph{universal factorization} for
the global p.d. kernel $K$. Hence if $\left(B,\mathscr{F}_{B},\mu,k^{\left(B\right)}\right)$
is a solution to (\ref{eq:k5}), then we get a contractive canonical
embedding
\[
L^{2}\left(B,\mu\right)\xrightarrow{\;W_{\left(B,\mu\right)}\;}L^{2}\left(\Omega,\mathscr{F},\mathbb{P}\right)
\]
by composition; see \thmref{A7} and \corref{T1}: 
\[
\xymatrix{L^{2}\left(B,\mu\right)\ar[rr]^{W_{\left(B,\mu\right)}=T_{K}T_{B}^{*}}\ar[dr]_{\underset{\text{see (\ref{eq:A11})}}{T_{B}^{*}}} &  & L^{2}\left(\Omega,\mathbb{P}\right)\\
 & \mathscr{H}\left(K\right)\ar[ru]_{\underset{\text{(see (\ref{eq:g5}))}}{T_{K}}}
}
\]

\begin{defn}
\label{def:k1}Let $K,L:X\times X\rightarrow\mathbb{C}$ be positive
definite (p.d.) kernels. An order relation $K\ll L$ is specified
as follows: For all finite subset $F=\left\{ x_{i}\right\} _{1}^{N}\subset X$,
and for all $\left(c_{i}\right)_{1}^{N}$, $c_{i}\in\mathbb{C}$,
\begin{equation}
\sum_{i}\sum_{j}\overline{c}_{i}c_{j}K\left(x_{i},x_{j}\right)\leq\sum_{i}\sum_{j}\overline{c}_{i}c_{j}L\left(x_{i},x_{j}\right).\label{eq:k1}
\end{equation}
\end{defn}

Given $K,L$ as above, assuming $K\ll L$, let $\mathscr{H}\left(K\right)$
and $\mathscr{H}\left(L\right)$ be the respective RKHS. 
\begin{defn}
Let $T:\mathscr{H}\left(L\right)\rightarrow\mathscr{H}\left(K\right)$
be the contraction by extension for $x\in X$: 
\begin{equation}
T:L\left(\cdot,x\right)\longmapsto K\left(\cdot,x\right)\label{eq:A5}
\end{equation}
where $L\left(\cdot,x\right)$ and $K\left(\cdot,x\right)$ are both
functions on $X$. 

Specifically, $T=T_{L\rightarrow K}$ extends by linearity and norm
closure
\begin{equation}
\left\Vert \sum\nolimits _{i}c_{i}K\left(\cdot,x_{i}\right)\right\Vert _{\mathscr{H}\left(K\right)}^{2}\leq\left\Vert \sum\nolimits _{i}c_{i}L\left(\cdot,x_{i}\right)\right\Vert _{\mathscr{H}\left(L\right)}^{2}.
\end{equation}
Hence $T:\mathscr{H}\left(L\right)\rightarrow\mathscr{H}\left(K\right)$
is norm contractive, i.e., with respect to the respective norms in
the two RKHSs.
\end{defn}

Given a positive definite (p.d.) kernel $K$, we shall consider families
of factorizations corresponding to certain admissible measures $\mu$.
Specifically, given $\mu$, we shall construct associated p.d. kernels
$K^{\left(\mu\right)}$. We shall then say that $\mu$ is admissible
iff (Def) $K^{\left(\mu\right)}\ll K$; where $\ll$ is the order
relation (on p.d. kernels) defined in \defref{k1}. The kernel $K^{\left(\mu\right)}$
is defined in (\ref{eq:k5}) below.
\begin{defn}
\label{def:FS}Let $K:X\times X\rightarrow\mathbb{C}$ be a p.d.
kernel. Set 
\begin{equation}
\mathscr{F}\left(K\right):=\left\{ B,\mathscr{F}_{B},\mu,k^{\left(B\right)}\mathrel{;}K\left(x,y\right)=\int_{B}\overline{k_{x}^{\left(B\right)}\left(b\right)}k_{y}^{\left(B\right)}\left(b\right)d\mu\left(b\right)\right\} ;\label{eq:k4}
\end{equation}
and by contrast, we set 
\begin{align}
\mathscr{F}S\left(K\right) & :=\Big\{ B,\mathscr{F}_{B},\mu,k^{\left(B\right)}\mathrel{;}K^{\left(\mu\right)}\left(x,y\right)=\int_{B}\overline{k_{x}^{\left(B\right)}\left(b\right)}k_{y}^{\left(B\right)}\left(b\right)d\mu\left(b\right),\label{eq:k5}\\
 & \qquad\text{and \ensuremath{K^{\left(\mu\right)}\ll K}}\Big\},\nonumber 
\end{align}
where $\left(B,\mathscr{F}_{B},\mu\right)$ is a measure space, and
$\{k_{x}^{\left(B\right)}\}_{x\in X}$ a system of functions in $L^{2}\left(\mu\right)$. 
\end{defn}

\begin{lem}
\label{lem:TB}Let $K:X\times X\rightarrow\mathbb{C}$ be a p.d. kernel
as above, and assume $\left(B,\mathscr{F}_{B},\mu,k^{\left(B\right)}\right)\in\mathscr{F}S\left(K\right)$,
then 
\begin{equation}
T_{B}:\mathscr{H}\left(K\right)\ni\sum_{i}c_{i}K\left(\cdot,x_{i}\right)\longmapsto\sum_{i}c_{i}k_{x_{i}}^{\left(B\right)}\left(\cdot\right)\in L^{2}\left(B,\mu\right)\label{eq:A9}
\end{equation}
is contractive.
\end{lem}

\begin{proof}
Let $F=\sum_{i}c_{i}K\left(\cdot,x_{i}\right)$, and set $T_{B}\left(F\right)=\widetilde{F}=\sum_{i}c_{i}k_{x_{i}}^{\left(B\right)}\left(\cdot\right)$
on $B$. Then 
\begin{align*}
\left\Vert T_{B}\left(F\right)\right\Vert _{L^{2}\left(\mu\right)}^{2} & =\int_{B}\left|\widetilde{F}\left(b\right)\right|^{2}d\mu\left(b\right)\\
 & =\int_{B}\left|\sum c_{i}k_{x_{i}}^{\left(B\right)}\left(b\right)\right|^{2}d\mu\left(b\right)\\
 & =\sum_{i}\sum_{j}\overline{c}_{i}c_{j}\int_{B}\overline{k_{x_{i}}^{\left(B\right)}\left(b\right)}k_{x_{j}}^{\left(B\right)}\left(b\right)d\mu\left(b\right)\\
 & =\sum_{i}\sum_{j}\overline{c}_{i}c_{j}K^{\left(\mu\right)}\left(x_{i},x_{j}\right)\\
 & \leq\sum_{i}\sum_{j}\overline{c}_{i}c_{j}K\left(x_{i},x_{j}\right)=\left\Vert F\right\Vert _{\mathscr{H}\left(K\right)}^{2},
\end{align*}
since $K^{\left(\mu\right)}\ll K$ which is assumed; see (\ref{eq:k5}).
The proof now finishes as usual with the use of norm-completion and
use of contractivity. 
\end{proof}
\begin{defn}
We say that $\left(B,\mathscr{F}_{B},\mu,k^{\left(B\right)}\right)$
is a \emph{boundary} for $K$ iff (Def.) (\ref{eq:k4}) holds. We
say that it is a \emph{sub-boundary} iff (Def.) $K^{\left(\mu\right)}\ll K$
where $K^{\left(\mu\right)}$ is defined as in (\ref{eq:k5}). See
also (\ref{eq:k1}).
\end{defn}

\begin{thm}
\label{thm:A7}Let $K:X\times X\rightarrow\mathbb{C}$ be given, assumed
positive definite. Let $\left(B,\mathscr{F}_{B},\mu\right)$ be a
fixed sigma-finite measure space, and set, for $\left(x,y\right)\in X\times X$,
\begin{equation}
K^{\left(\mu\right)}\left(x,y\right)=\int_{B}\overline{k_{x}^{\left(B\right)}\left(b\right)}k_{y}^{\left(B\right)}\left(b\right)d\mu\left(b\right),\label{eq:A10}
\end{equation}
where $\{k_{x}^{\left(B\right)}\}_{x\in X}$ is a subset in $L^{2}\left(B,\mu\right)$.
Then the following conditions are equivalent:
\begin{enumerate}
\item \label{enu:mu1}$K^{\left(\mu\right)}\ll K$
\item \label{enu:mu2}The operator $T_{B}$ in (\ref{eq:A9}) is well defined
and contractive.
\item \label{enu:mu3}For $\varphi\in L^{2}\left(B,\mu\right)$ and $x\in X$,
set 
\begin{equation}
\left(S_{B}\varphi\right)\left(x\right)=\int_{B}\overline{k_{x}^{\left(B\right)}\left(b\right)}\varphi\left(b\right)d\mu\left(b\right);\label{eq:A11}
\end{equation}
then $S_{B}$ is well defined and contractive $S_{B}:L^{2}\left(B,\mu\right)\rightarrow\mathscr{H}\left(K\right)$.
\end{enumerate}
Moreover, if the conditions are satisfied, then 
\begin{equation}
S_{B}=T_{B}^{*}\label{eq:A12}
\end{equation}
holds, where the adjoint in (\ref{eq:A12}) refers to the two Hilbert
spaces $\mathscr{H}\left(K\right)$ and $L^{2}\left(B,\mu\right)$
with the respective inner products.
\end{thm}

\begin{proof}
(\ref{enu:mu1})$\Rightarrow$(\ref{enu:mu2}). This is \lemref{TB}.

(\ref{enu:mu2})$\Rightarrow$(\ref{enu:mu3}). For this we first
show that (\ref{enu:mu2}) yields a function in $\mathscr{H}\left(K\right)$
for all $\varphi\in L^{2}\left(B,\mu\right)$, i.e., the function
$X\ni x\longmapsto\left\langle k_{x}^{\left(B\right)},\varphi\right\rangle _{L^{2}\left(\mu\right)}$.
We refer to \lemref{A1} for this purpose. So let $\varphi\in L^{2}\left(B,\mu\right)$.
Fix $N$, and $\left\{ x_{i}\right\} _{i=1}^{N}\subset X$; then if
$c_{i}\in\mathbb{C}$, $i=1,\dots,N$, then 
\begin{eqnarray*}
\left|\sum_{i=1}^{N}c_{i}\left\langle k_{x_{i}}^{\left(B\right)},\varphi\right\rangle _{L^{2}\left(\mu\right)}\right|^{2} & = & \left|\left\langle \sum_{i=1}^{N}c_{i}k_{x_{i}}^{\left(B\right)}\left(\cdot\right),\varphi\left(\cdot\right)\right\rangle _{L^{2}\left(\mu\right)}\right|^{2}\\
 & \underset{\text{Schwarz}}{\leq} & \left\Vert \sum_{i=1}^{N}c_{i}k_{x_{i}}^{\left(B\right)}\right\Vert _{L^{2}\left(\mu\right)}^{2}\left\Vert \varphi\right\Vert _{L^{2}\left(\mu\right)}^{2}\\
 & = & \sum_{i=1}^{N}\sum_{j=1}^{N}\overline{c}_{i}c_{j}\left\langle k_{x_{i}}^{\left(B\right)},k_{x_{j}}^{\left(B\right)}\right\rangle \left\Vert \varphi\right\Vert _{L^{2}\left(\mu\right)}^{2}\\
 & = & \sum_{i=1}^{N}\sum_{j=1}^{N}\overline{c}_{i}c_{j}K^{\left(\mu\right)}\left(x_{i},x_{j}\right)\left\Vert \varphi\right\Vert _{L^{2}\left(\mu\right)}^{2}\\
 & \underset{\text{by \ensuremath{\left(\ref{enu:mu2}\right)}}}{\leq} & \sum_{i=1}^{N}\sum_{j=1}^{N}\overline{c}_{i}c_{j}K\left(x_{i},x_{j}\right)\left\Vert \varphi\right\Vert _{L^{2}\left(\mu\right)}^{2}.
\end{eqnarray*}
The desired conclusion now follows. Moreover, for $\varphi\in L^{2}\left(\mu\right)$,
$x\in X$, we have 
\[
\left\langle K\left(\cdot,x\right),T_{B}^{*}\varphi\right\rangle _{\mathscr{H}\left(X\right)}=\left\langle k_{x}^{\left(B\right)},\varphi\right\rangle _{L^{2}\left(\mu\right)}=\int_{B}\overline{k_{x}\left(b\right)}\varphi\left(b\right)d\mu\left(b\right),
\]
and the conclusion (\ref{enu:mu3}) now follows.

The remaining conclusions, including (\ref{enu:mu3})$\Rightarrow$(\ref{enu:mu1}),
now follow from one more application of \lemref{A1}. For the operator
norms, we have
\[
\left\Vert T_{B}\right\Vert _{\mathscr{H}\left(X\right)\rightarrow L^{2}\left(\mu\right)}=\left\Vert S_{B}\right\Vert _{L^{2}\left(\mu\right)\rightarrow\mathscr{H}\left(X\right)}\leq1.
\]
\end{proof}

\begin{rem}
\thmref{A7} offers a way to get non-trivial determinantal point processes
from our setting; i.e., from (\ref{eq:k5}). From a solution to (\ref{eq:k5})
we get the contractive operator $T=T_{B}$ from $\mathscr{H}\left(K\right)$
into $L^{2}\left(B,\mu\right)$. See \thmref{A7}. Then the operator
$L:=TT^{*}$ is contractive in the Hilbert space $L^{2}\left(B,\mu\right)$,
which is precisely the condition typically imposed when dealing with
determinantal processes (see, e.g., \cite{MR2552864,kulesza2012determinantal,2019arXiv190304945K}).
Of course then we get determinantal processes with probability measures
on point configurations in $B$; not in $X$.
\end{rem}

\begin{cor}
\label{cor:T1}If one of the conditions in \thmref{A7} holds, then
the operator $T_{B}^{*}T_{B}:\mathscr{H}\left(K\right)\rightarrow\mathscr{H}(K^{\left(\mu\right)})$
is contractive, and 
\begin{equation}
K^{\left(\mu\right)}\left(x,y\right)=\left(T_{B}^{*}T_{B}K_{y}\right)\left(x\right)
\end{equation}
for all $x,y\in X$. (See also (\ref{eq:A5}).)
\end{cor}

\begin{proof}
One checks that 
\begin{align*}
\left(T_{B}^{*}T_{B}K_{y}\right)\left(x\right) & =\left\langle K_{x},T_{B}^{*}T_{B}K_{y}\right\rangle _{\mathscr{H}\left(K\right)}\\
 & =\left\langle T_{B}K_{x},T_{B}K_{y}\right\rangle _{L^{2}\left(\mu\right)}\\
 & =\int_{B}\overline{k_{x}^{\left(B\right)}\left(b\right)}k_{y}^{\left(B\right)}\left(b\right)d\mu\left(b\right)\\
 & =K^{\left(\mu\right)}\left(x,y\right)=K_{y}^{\left(\mu\right)}\left(x\right).
\end{align*}
\end{proof}

\subsection{Resistance Networks and Energy Hilbert Spaces}

A factorization as in (\ref{eq:k4}) occurs naturally in the study
of resistance networks and the corresponding energy Hilbert spaces
(see, e.g., \cite{MR3860446,MR2735315,MR3051696,MR3630401,MR3842203}). 

Let $\left(V,E\right)$ be a countably infinite network with vertices
$V$ and edges $E$. Assume it is locally finite, i.e., the set of
nearest neighbors $N\left(x\right)=\left\{ y\in V\mid y\sim x\right\} $
is finite for all $x\in V$, where $y\sim x$ iff $\left(x,y\right)\in E$.
Assume further that $\left(V,E\right)$ is connected, i.e., for all
$x,y\in V$, there exist $n\in\mathbb{N}$ and vertices $\left(x_{i}\right)_{i=0}^{n}$,
such that $x_{0}=x$, $\left(x_{i},x_{i+1}\right)\in E$ and $x_{n}=y$. 

Fix a base point $o\in V$, and a conductance function $c:E\rightarrow\mathbb{R}_{\geq0}$.
Let $\mathscr{H}_{E}$ be the energy Hilbert space of all functions
$f:V\rightarrow\mathbb{C}$ satisfying 
\begin{equation}
\left\Vert f\right\Vert _{\mathscr{H}_{E}}^{2}=\frac{1}{2}\underset{\left(x,y\right)\in E}{\sum\sum}c_{xy}\left|f\left(x\right)-f\left(y\right)\right|^{2}<\infty.\label{eq:B16}
\end{equation}

\begin{defn}
\label{def:rn}We say $\left(V,E,c\right)$ is a \emph{resistance
network}, where 
\begin{itemize}
\item $V$: a discrete set of vertices;
\item $E$: a fixed set of edges, so $E$ contained in $V\times V\backslash\left(\text{the diagonal}\right)$,
and locally finite.
\item $c$: a fixed function on $E$, representing conductance ($=1/\text{resistance}$). 
\end{itemize}
\end{defn}

The connectedness assumption implies that, for all $x,y\in V$, there
exists a constant $C_{x,y}\geq0$, and 
\begin{equation}
\left|f\left(x\right)-f\left(y\right)\right|\leq C_{x,y}\left\Vert f\right\Vert _{\mathscr{H}_{E}}.
\end{equation}
See \figref{res}.

\begin{figure}[H]
\includegraphics[width=0.3\columnwidth]{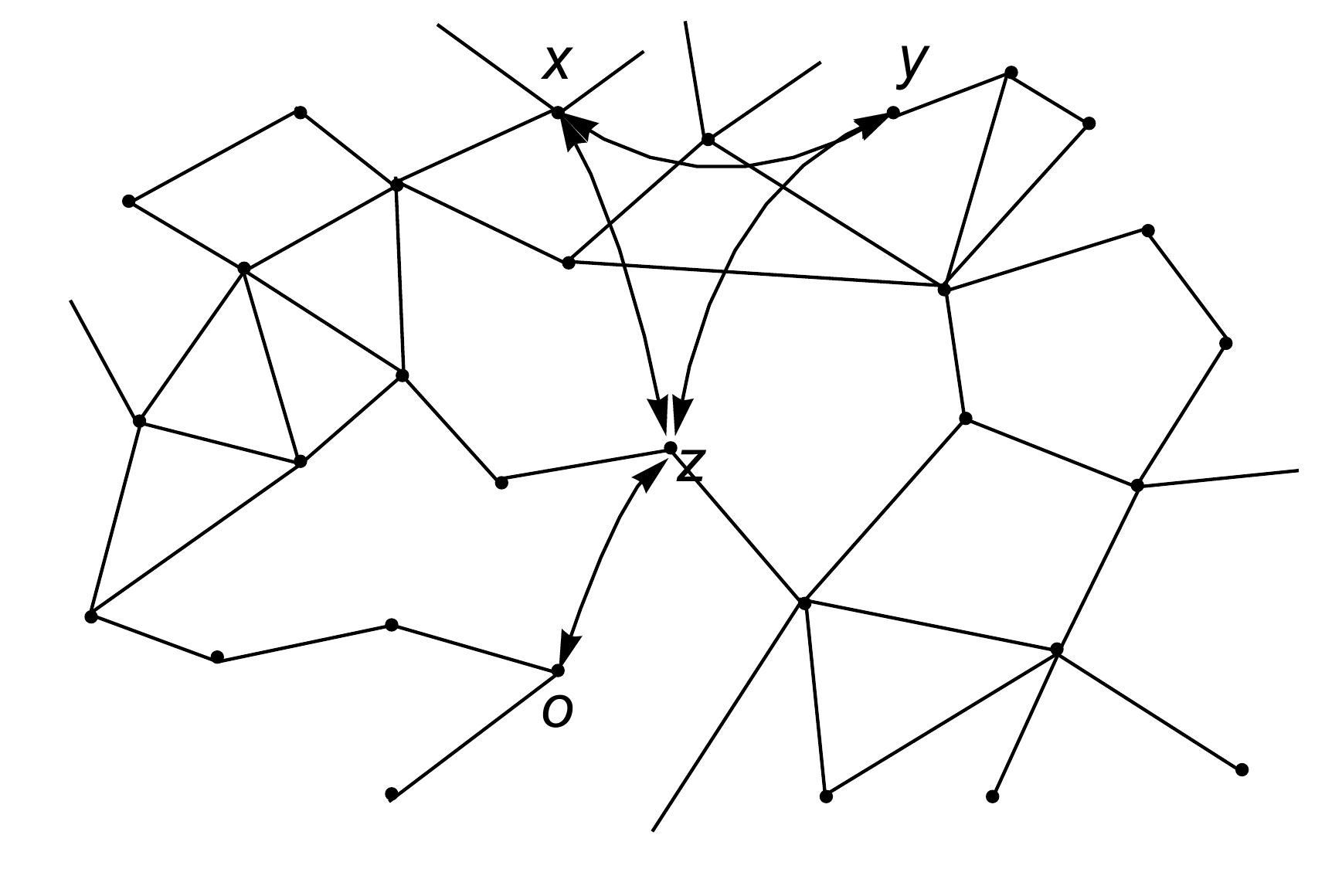}

\caption{\label{fig:res}Current flows in a connected resistance network.}

\end{figure}

By Riesz, there exists a unique $v_{x}\left(=v_{x,o}\right)\in\mathscr{H}_{E}$,
such that 
\begin{equation}
f\left(x\right)-f\left(o\right)=\left\langle v_{x},f\right\rangle _{\mathscr{H}_{E}}.\label{eq:B18}
\end{equation}
Set 
\begin{equation}
G\left(x,y\right)=\left\langle v_{x},v_{y}\right\rangle _{\mathscr{H}_{E}}=v_{y}\left(x\right)-v_{y}\left(o\right).
\end{equation}
It may be normalized by setting $v_{x}\left(o\right)=0$, for all
$x\in V$. Then, $G\left(x,y\right)=v_{y}\left(x\right)$. 

The following consideration, for the energy Hilbert space $\mathscr{H}_{E}$
in \defref{rn}, shows that the \emph{contractivity} in \thmref{A7}
is a non-trivial restriction. Specifically, let $\left(V,E,c\right)$
be as in \defref{rn} and let $\mathscr{H}_{E}$ be the energy Hilbert
space defined from (\ref{eq:B16}). Then the function 
\begin{equation}
c\left(x\right):=\sum_{y\sim x}c_{xy}\label{eq:B20}
\end{equation}
takes finite values for all $x\in V$. Moreover, the Dirac function
$\delta_{x}$, 
\[
\delta_{x}\left(y\right)=\begin{cases}
1 & y=x\\
0 & y\neq x
\end{cases}
\]
on $V$ satisfies 
\begin{equation}
\delta_{x}\in\mathscr{H}_{E},\;\text{and }\:\left\Vert \delta_{x}\right\Vert _{\mathscr{H}_{E}}^{2}=c\left(x\right).
\end{equation}
Since the function $c$ (see (\ref{eq:B20})) is generally unbounded,
it follows that the inclusion mapping
\begin{equation}
l^{2}\left(V\right)\supset span\left\{ \delta_{x}\right\} \ni\delta_{x}\xrightarrow{\;J\;}\delta_{x}\in\mathscr{H}_{E}
\end{equation}
is unbounded. 

Nonetheless, the operator 
\[
JJ^{*}:\mathscr{H}_{E}\longrightarrow\mathscr{H}_{E}
\]
is well defined. One checks that for functions $f$ on $span_{x\in V}\left\{ v_{x}\right\} $
(see (\ref{eq:B18})) we~get 
\begin{equation}
\left(JJ^{*}f\right)\left(x\right)=\sum_{y\sim x}c_{xy}\left(f\left(x\right)-f\left(y\right)\right).\label{eq:B23}
\end{equation}
The operator defined in (\ref{eq:B23}) is also unbounded; it is called
the \emph{graph-Laplacian}, written $\Delta_{E}:=JJ^{*}$. 
\begin{example}[Nearest neighbors]
Consider $V=\mathbb{Z}_{\geq0}$, $E=\left\{ \left(i,i+1\right)\right\} $,
$c_{i,i+1}=1$, and the base point at $i=0$. In this case, $\mathscr{H}_{E}$
consists of functions satisfying 
\begin{equation}
\left\Vert f\right\Vert _{\mathscr{H}_{E}}^{2}=\sum_{i=0}^{\infty}\left|f\left(i+1\right)-f\left(i\right)\right|^{2}<\infty.
\end{equation}
A solution $\left(B,\mathscr{F}_{B},\mu,k^{\left(B\right)}\right)$
to (\ref{eq:k4}) is as follows: 

Let $k_{i}^{\left(B\right)}\left(x\right)=\chi_{\left[0,i\right]}\left(x\right)$,
$i\in\mathbb{Z}_{\geq0}$, $x\in\mathbb{R}_{\geq0}$, and $\mu=$
Lebesgue measure. Then, 
\begin{align*}
\int_{0}^{\infty}\overline{k_{i}^{\left(B\right)}\left(x\right)}k_{j}^{\left(B\right)}\left(x\right)d\mu\left(x\right) & =\int_{0}^{\infty}\chi_{\left[0,i\right]}\left(x\right)\chi_{\left[0,j\right]}\left(x\right)d\mu\left(x\right)\\
 & =\mu\left(\left[0,i\right]\cap\left[0,j\right]\right)\\
 & =i\wedge j\\
 & =\left\langle v_{i},v_{j}\right\rangle _{\mathscr{H}_{E}}\\
 & =G\left(i,j\right)\\
 & =K\left(i,j\right).
\end{align*}
\end{example}

\subsection{Product Boundaries}

We note that, as a consequence of \thmref{A7}, we get the following
result for products of p.d. functions, and their corresponding boundaries;
see \defref{FS}, (\ref{eq:k4}) and (\ref{eq:k5}).
\begin{cor}
Let $K_{i}:X\times X\rightarrow\mathbb{C}$, $i=1,2$, be positive
definite functions. Let $\left(B_{i},\mathscr{F}_{B_{i}},\mu_{i},k^{\left(B_{i}\right)}\right)$,
$i=1,2$, be as specified in \defref{FS}. Then if each $\left(B_{i},\mathscr{F}_{B_{i}},\mu_{i},k^{\left(B_{i}\right)}\right)$
satisfies (\ref{eq:k4}), respectively (\ref{eq:k5}) relative to
$K_{i}$, $i=1,2$, then $B_{1}\times B_{2}$, $\mathscr{F}_{B_{1}\times B_{2}}$,
$\mu_{1}\times\mu_{2}$, $k_{x}^{\left(B_{1}\right)}k_{x}^{\left(B_{2}\right)}$,
will also satisfy (\ref{eq:k4}), respectively (\ref{eq:k5}), now
w.r.t. the product p.d. function 
\begin{equation}
\left(K_{1}K_{2}\right)\left(x,y\right)=K_{1}\left(x,y\right)K_{2}\left(x,y\right),\quad\text{for }\left(x,y\right)\in X\times X.\label{eq:pd1}
\end{equation}
 
\end{cor}

\begin{proof}
First note that if $\mathscr{L}_{i}$, $i=1,2$, are Hilbert spaces
and $l_{i}:X\rightarrow\mathscr{L}_{i}$, $i=1,2$, functions such
that 
\begin{equation}
K_{i}\left(x,y\right)=\left\langle l_{i}\left(x\right),l_{i}\left(y\right)\right\rangle _{\mathscr{L}_{i}},\quad i=1,2,
\end{equation}
and for $K$ in (\ref{eq:pd1}) we get 
\begin{equation}
K\left(x,y\right)=\left\langle l_{1}\left(x\right)\otimes l_{2}\left(x\right),l_{1}\left(y\right)\otimes l_{2}\left(y\right)\right\rangle _{\mathscr{L}_{1}\otimes\mathscr{L}_{2}}\label{eq:pd3}
\end{equation}
where we use the tensor inner product of $\mathscr{L}_{1}\otimes\mathscr{L}_{2}$
on the RHS in (\ref{eq:pd3})

We may now apply \thmref{A7} to this with 
\begin{equation}
k_{x}^{\left(B_{1}\times B_{2}\right)}\left(b_{1},b_{2}\right)=k_{x}^{\left(B_{1}\right)}\left(b_{1}\right)k_{x}^{\left(B_{2}\right)}\left(b_{2}\right),\quad\text{for }\left(b_{1},b_{2}\right)\in B_{1}\times B_{2}.
\end{equation}
The desired conclusion follows from consideration of the operator
$T_{B_{i}}:\mathscr{H}\left(K_{i}\right)\rightarrow L^{2}\left(B_{i},\mu_{i}\right)$,
$i=1,2$, from (\ref{eq:A9}) in \lemref{TB}.

The relevant implications are as follows: If each $T_{B_{i}}$ is
isometric (resp., contractive), then $T_{B_{1}}\otimes T_{B_{2}}:\mathscr{H}\left(K\right)\rightarrow L^{2}\left(B_{1}\times B_{2},\mu_{1}\times\mu_{2}\right)$
is also isometric (resp., contractive.)
\end{proof}

\section{Frames and Associated Kernels}

Starting with a frame in Hilbert space, we discuss associated positive
definite kernels $K$ and their boundary representations, in the sense
of \defref{FS}.
\begin{defn}
\label{def:C1}Let $\mathscr{H}$ be a separable Hilbert space. A
system of vectors $\left\{ \varphi_{n}\right\} _{n\in\mathbb{N}}$
is said to be a \emph{frame} in $\mathscr{H}$, if there exist constants
$0<A\leq B<\infty$, such that 
\begin{equation}
A\left\Vert x\right\Vert ^{2}\leq\sum_{n}\left|\left\langle \varphi_{n},x\right\rangle \right|^{2}\leq B\left\Vert x\right\Vert ^{2}\label{eq:C1}
\end{equation}
holds for all $x\in\mathscr{H}$. If $A=B=1$, the system $\left\{ \varphi_{n}\right\} _{n\in\mathbb{N}}$
is called a \emph{Parseval frame}.
\end{defn}

\begin{cor}
Let $\left(B,\mathscr{F}_{B},\mu,k^{\left(B\right)}\right)$ and $K^{\left(\mu\right)}$
be as in (\ref{eq:k5}). In particular, $K^{\left(\mu\right)}\ll K$
(also see (\ref{eq:A10})). Let $T_{B}:\mathscr{H}\left(K\right)\rightarrow L^{2}\left(B,\mu\right)$
be the contraction from \lemref{TB}, and $T_{B}^{*}$ be the adjoint
in (\ref{eq:A11}). 

Let $T_{B}=V\left|T_{B}\right|$ be the corresponding polar decomposition,
i.e., $\left|T_{B}\right|=\left(T_{B}^{*}T_{B}\right)^{1/2}$ and
$V$ is the unique partial isometry from $\overline{ran\left(T_{B}^{*}\right)}$
onto $\overline{ran\left(T_{B}\right)}$. Suppose $\left\{ e_{n}\right\} _{n\in\mathbb{N}}$
is an orthonormal basis (ONB) in $L^{2}\left(B,\mu\right)$, then
$\left\{ V^{*}e_{n}\right\} _{n\in\mathbb{N}}$ is a Parseval frame
in $\mathscr{H}\left(K\right)\ominus ker\left(T_{B}\right)$. 
\end{cor}

\begin{proof}[Proof (sketch)]
 Note that $V=T_{B}\left(T_{B}^{*}T_{B}\right)^{-1/2}$ when restricted
to $\mathscr{H}\left(K\right)\ominus ker\left(T_{B}\right)$. Then,
for all $F,G\in\mathscr{\mathscr{H}\left(K\right)\ominus}ker\left(T_{B}\right)$,
we have 
\begin{align*}
\left\langle F,G\right\rangle _{\mathscr{H}\left(K\right)} & =\left\langle VF,VG\right\rangle _{L^{2}\left(\mu\right)}\\
 & =\sum_{n}\left\langle VF,e_{n}\right\rangle _{L^{2}\left(\mu\right)}\left\langle e_{n},VG\right\rangle _{L^{2}\left(\mu\right)}\\
 & =\sum_{n}\left\langle F,V^{*}e_{n}\right\rangle _{\mathscr{H}\left(K\right)}\left\langle V^{*}e_{n},G\right\rangle _{\mathscr{H}\left(K\right)},
\end{align*}
and the assertion follows. 
\end{proof}

\begin{lem}
Let $K,L:X\times X\rightarrow\mathbb{C}$ be p.d. kernels such that
$K\ll L$. Let $\mathscr{H}\left(K\right)$, $\mathscr{H}\left(L\right)$
be the respective RKHSs. 

Suppose $\mathscr{H}\left(K\right)$ is separable, then there exists
a Hilbert space $\mathscr{K}$ and a contraction $T:\mathscr{H}\left(L\right)\rightarrow\mathscr{K}$,
such that 
\[
K_{y}\left(x\right)=\left(T^{*}TL_{y}\right)\left(x\right),\quad\forall x,y\in X.
\]
\end{lem}

\begin{proof}
Choose a Parseval frame $\left\{ \varphi_{n}\right\} $ in $\mathscr{H}\left(K\right)$.
Then, in particular, 
\begin{align*}
K\left(x,y\right) & =\left\langle K_{x},K_{y}\right\rangle _{\mathscr{H}\left(K\right)}\\
 & =\sum_{n}\left\langle K_{x},\varphi_{n}\right\rangle _{\mathscr{H}\left(K\right)}\left\langle \varphi_{n},K_{y}\right\rangle _{\mathscr{H}\left(K\right)}\\
 & =\sum_{n}\varphi_{n}\left(x\right)\overline{\varphi_{n}\left(y\right)}.
\end{align*}
For all $L_{y}\left(\cdot\right)=L\left(\cdot,y\right)\in\mathscr{H}\left(L\right)$,
set 
\[
TL_{y}=\left(\overline{\varphi_{n}\left(y\right)}\right)_{n\in\mathbb{N}}\in l^{2}.
\]
Since 
\[
\left\Vert TL_{y}\right\Vert _{l^{2}}^{2}=\sum_{n}\left|\varphi_{n}\left(y\right)\right|^{2}=K\left(y,y\right)\leq L\left(y,y\right)=\left\Vert L_{y}\right\Vert _{\mathscr{H}\left(L\right)}^{2},
\]
then $T$ extends to a contraction, $\mathscr{H}\left(L\right)\rightarrow l^{2}$.

One checks that, the adjoint $T^{*}$ is given by 
\[
\left(T^{*}\xi\right)\left(x\right)=\sum_{n}\varphi_{n}\left(x\right)\xi_{n},\quad\forall\xi\in l^{2}.
\]
Then, it follows that 
\[
\left(T^{*}TL_{y}\right)\left(x\right)=\sum_{n}\varphi_{n}\left(x\right)\overline{\varphi_{n}\left(y\right)}=K_{y}\left(x\right),\quad\forall x,y\in X.
\]

Therefore the assertion follows with $\mathscr{K}=l^{2}$. 
\end{proof}
If (\ref{eq:C1}) in \defref{C1} has $A=B=1$, then we say that $\left\{ \varphi_{n}\right\} _{n\in\mathbb{N}}$
is a Parseval frame.
\begin{thm}
Let $X\times X\xrightarrow{\;K\;}\mathbb{C}$ be a p.d. kernel, and
let $\mathscr{H}\left(K\right)$ be the corresponding reproducing
kernel Hilbert space (RKHS). Assume $\left\{ \varphi_{n}\right\} _{n\in\mathbb{N}}\subset\mathscr{H}\left(K\right)$
is a Parseval frame in $\mathscr{H}\left(K\right)$, then the following
two assertions hold:
\begin{enumerate}
\item \label{enu:Cg1}For all $\left(x,y\right)\in X\times X$, one has
\[
K\left(x,y\right)=\sum_{n\in\mathbb{N}}\varphi_{n}\left(x\right)\overline{\varphi_{n}\left(y\right)},
\]
and so a pair $\left(\mathbb{N},\text{counting measure}\right)$ satisfies
the boundary condition (\ref{eq:k4}) in \defref{FS}.
\item \label{enu:Cg2}Let $\left\{ Z_{n}\right\} _{n\in\mathbb{N}}$ be
an i.i.d. system of $N\left(0,1\right)$ random variables, i.e., a
system of independent, identically distributed standard Gaussians.
Then 
\begin{equation}
V_{x}\left(\cdot\right)=\sum_{n\in\mathbb{N}}\varphi_{n}\left(x\right)Z_{n}\left(\cdot\right)\label{eq:C2}
\end{equation}
is a Gaussian process, indexed by $X$, and with covariance kernel
\[
\mathbb{E}\left(V_{x}\overline{V}_{y}\right)=K\left(x,y\right),\quad\forall\left(x,y\right)\in X\times X;
\]
see also (\ref{eq:g4}).
\end{enumerate}
\end{thm}

\begin{proof}[Proof (sketch)]
 We already proved (\ref{enu:Cg1}) as part of \thmref{A7}. The
proof of (\ref{enu:Cg2}) follows from the i.i.d. property for the
choice of $\left\{ Z_{n}\right\} _{n\in\mathbb{N}}$ system; also
called a Monte Carlo simulation; see e.g., \cite{doi:10.1002/9781119414421.ch2,AD92,MR3526117}.
\end{proof}

\section{The Drury-Arveson Kernel}

The Drury-Arveson kernel is defined on the unit ball in $k$ complex
dimensions; or more generally on the unit ball in a specified Hilbert
space; see (\ref{eq:k18}) below. It has been studied extensively,
first in \cite{MR480362,MR1668582}, and in \cite{MR2254502,MR3721329,Jorgensen2017,MR3526117}.
A key result from \cite{MR1668582} which motivated our present analysis
is the following: If $k>1$, then the corresponding Drury-Arveson
kernel does not admit a boundary space which is contained in $\mathbb{C}^{k}$.

The discussion above further motivates a recent renewed interest in
the Drury-Arveson kernel and its harmonic analysis. In addition to
the papers cited, we add, \cite{SA19,MR3780504,MR3912786,MR3834659,MR3797923,MR3691998}.
\begin{example}[Drury-Arveson kernel; see e.g., \cite{MR1668582}]
\label{exa:D1} Let 
\begin{align}
X=B_{k} & :=\left\{ z=\left(z_{i}\right)\in\mathbb{C}^{k}\mathrel{;}\sum\nolimits _{i=1}^{k}\left|z_{i}\right|^{2}<1\right\} ,\label{eq:B1}\\
\partial B_{k} & :=\left\{ z=\left(z_{i}\right)\in\mathbb{C}^{k}\mathrel{;}\sum\nolimits _{i=1}^{k}\left|z_{i}\right|^{2}=1\right\} .\label{eq:B2}
\end{align}
Note 
\begin{equation}
\partial B_{k}\simeq U_{k}\big/U_{k-1}
\end{equation}
where $U_{k}:=$ all unitary $k\times k$ matrices on $\mathbb{C}$.
Pick the normalized Haar measure $\lambda_{k}$ on $U_{k}$ and set
\begin{equation}
\mu_{k}:=\lambda_{k}\circ\pi_{k}^{-1}\label{eq:k10}
\end{equation}
where 
\begin{equation}
\pi_{k}:U_{k}\rightarrow U_{k}\big/U_{k-1}\simeq\partial B_{k}.\label{eq:k11}
\end{equation}
We note that $U_{k}$ acts on $\partial B_{k}$ via (\ref{eq:k11})
and that $\mu_{k}$ (see (\ref{eq:k10})) is $U_{k}$ invariant, i.e.,
we have a compact homogeneous space for the Lie group $U_{k}$. 

Let $\left\langle \cdot,\cdot\right\rangle _{\mathbb{C}^{k}}$ denote
the standard inner product in $\mathbb{C}^{k}$, i.e., $\left\langle z,w\right\rangle _{\mathbb{C}^{k}}:=\sum_{j=1}^{k}\overline{z}_{j}w_{k}$,
and set 
\begin{equation}
K\left(z,w\right)=K_{DA}\left(z,w\right):=\frac{1}{1-\left\langle z,w\right\rangle _{\mathbb{C}^{k}}},\quad\left(z,w\right)\in B_{k}\times B_{k},\label{eq:B6}
\end{equation}
and 
\begin{equation}
K_{z}^{\left(\mu\right)}\left(b\right):=\frac{1}{1-\left\langle z,b\right\rangle _{\mathbb{C}^{k}}},\quad z\in B_{k},\:b\in\partial B_{k};\label{eq:B7}
\end{equation}
see (\ref{eq:B1}) and (\ref{eq:B2}).

We shall show that the assignment 
\begin{equation}
K\left(z,\cdot\right)\longmapsto\left(K_{z}^{\left(\mu\right)}\;\text{on \ensuremath{\partial B_{k}}}\right)\label{eq:B8}
\end{equation}
extends to a linear contractive operator $T_{B}:\mathscr{H}\left(K\right)\rightarrow L^{2}\left(\partial B_{k},\mu\right)$
with dense range in $L^{2}\left(\partial B_{k},\mu\right)$; see \corref{B3},
below.

Let $\mathscr{H}\left(K\right)$ be the symmetric Fock space, i.e.,
\begin{equation}
\mathscr{H}\left(K\right)=\mathbb{C}\oplus\left(\sum_{1}^{\infty}\otimes_{1}^{n}\mathbb{C}^{k}\right)^{\text{sym}}.
\end{equation}
Every $F\in\mathscr{H}\left(K\right)$ has a unique representation
\begin{equation}
F\left(z\right)=\sum_{n=0}^{\infty}\left\langle \xi_{n},z^{\otimes n}\right\rangle _{\mathscr{H}_{n}^{\text{sym}}},\quad z\in\mathbb{C}^{k}\label{eq:k7}
\end{equation}
where $\left(\xi_{n}\right)\in\left(\otimes_{1}^{n}\mathbb{C}^{k}\right)^{\text{sym}}=:\mathscr{H}_{n}^{\text{sym}}$,
and $\xi_{0}\in\mathbb{C}=:\mathscr{H}_{0}^{\text{sym}}$. Moreover,
\begin{equation}
\left\Vert F\right\Vert _{\mathscr{H}\left(K\right)}^{2}=\sum_{n=0}^{\infty}\left\Vert \xi_{n}\right\Vert _{\mathscr{H}_{n}^{\text{sym}}}^{2}
\end{equation}
where $\left\Vert \cdot\right\Vert _{\mathscr{H}_{n}^{\text{sym}}}$
is the usual Hilbert norm in the symmetric tensor space of order $n$. 
\end{example}

\begin{thm}
\label{thm:B2}Let $k$ be fixed, and let 
\begin{equation}
K\left(z,w\right):=\frac{1}{1-\left\langle z,w\right\rangle _{\mathbb{C}^{k}}},\quad\left(z,w\right)\in B_{k}\times B_{k}\label{eq:k18}
\end{equation}
be the Drury-Arveson kernel. Set
\begin{equation}
K^{\left(\mu\right)}\left(z,w\right)=\int_{\partial B_{k}}\overline{\frac{1}{1-\left\langle b,z\right\rangle _{\mathbb{C}^{k}}}}\frac{1}{1-\left\langle b,w\right\rangle _{\mathbb{C}^{k}}}d\mu\left(b\right),\label{eq:k19}
\end{equation}
then 
\[
K^{\left(\mu\right)}\ll K.
\]
\end{thm}

\begin{proof}
Define the Drury-Arveson kernel as in (\ref{eq:k18}), and set 
\begin{equation}
k_{z}^{\left(B\right)}\left(b\right)=\frac{1}{1-\left\langle z,b\right\rangle _{\mathbb{C}^{k}}},\quad b\in\partial B_{k},\;z\in B_{k}.\label{eq:k20}
\end{equation}
Note that $k_{z}^{\left(B\right)}\left(b\right)$ is well defined,
since 
\begin{align*}
\int_{B}\left|k_{z}^{\left(B\right)}\left(b\right)\right|^{2}d\mu\left(b\right) & =\sum_{n\geq0}\sum_{m\geq0}\int_{\partial B_{k}}\overline{\left\langle z,b\right\rangle _{\mathbb{C}^{k}}^{n}}\left\langle z,b\right\rangle _{\mathbb{C}^{k}}^{m}d\mu\left(b\right)\\
 & =\sum_{n\geq0}\sum_{m\geq0}\int_{\partial B_{k}}\overline{\left\langle z^{\otimes n},b^{\otimes n}\right\rangle }_{\mathscr{H}_{n}^{\text{sym}}}\left\langle z^{\otimes m},b^{\otimes m}\right\rangle _{\mathscr{H}_{n}^{\text{sym}}}d\mu\left(b\right)\\
 & =\sum_{n=0}^{\infty}\int_{\partial B_{k}}\left|\overline{\left\langle z^{\otimes n},b^{\otimes n}\right\rangle }_{\mathscr{H}_{n}^{\text{sym}}}\right|^{2}d\mu\left(b\right)\\
 & \leq\sum_{n=0}^{\infty}\left\Vert z^{\otimes n}\right\Vert _{\mathscr{H}_{n}^{\text{sym}}}^{2}=\frac{1}{1-\left\Vert z\right\Vert ^{2}}<\infty,\quad\text{as \ensuremath{\left\Vert z\right\Vert <1}. }
\end{align*}

In the DA-exmaple, we have $K^{\left(\mu\right)}\ll K$ where $K^{\left(\mu\right)}$
is defined as in (\ref{eq:k19}) (also see (\ref{eq:k5})) and $k^{\left(B\right)}$
as in (\ref{eq:k20}). Details:

For $c_{i}\in\mathbb{C}$, $z_{i}\in B_{k}$, $1\leq i\leq N$, we
have 
\begin{eqnarray*}
\sum_{i}\sum_{j}\overline{c}_{i}c_{j}K^{\left(\mu\right)}\left(z_{i},z_{j}\right) & = & \sum_{i}\sum_{j}\overline{c}_{i}c_{j}\int_{\partial B_{k}}\frac{1}{\overline{1-\left\langle z_{i},b\right\rangle _{\mathbb{C}_{k}}}}\frac{1}{1-\left\langle z_{j},b\right\rangle _{\mathbb{C}_{k}}}d\mu\left(b\right)\\
 & = & \int_{\partial B_{k}}\left|\sum\nolimits _{j}c_{j}\frac{1}{1-\left\langle z_{j},b\right\rangle _{\mathbb{C}_{k}}}\right|^{2}d\mu\left(b\right)\\
 & \leq & \left\Vert \sum\nolimits _{j}c_{j}\frac{1}{1-\left\langle \cdot,z_{j}\right\rangle _{\mathbb{C}^{k}}}\right\Vert _{\mathscr{H}\left(K\right)}^{2}\\
 & = & \sum_{i}\sum_{j}\overline{c}_{i}c_{j}\frac{1}{1-\left\langle z_{i},z_{j}\right\rangle _{\mathbb{C}^{k}}}\\
 & \underset{\text{see \ensuremath{\left(\ref{eq:k18}\right)}}}{=} & \sum_{i}\sum_{j}\overline{c}_{i}c_{j}K\left(z_{i},z_{j}\right).
\end{eqnarray*}
\end{proof}

\begin{cor}
\label{cor:B3} Let $K$ be the Drury-Arveson kernel, and let $T_{B}:\mathscr{H}\left(K\right)\rightarrow L^{2}\left(\partial B_{k},\mu\right)$
be the canonical contractive operator from (\ref{eq:B8}), then 
\[
Ran\left(T_{B}\right)\left(=T_{B}\left(\mathscr{H}\left(K\right)\right)\right)
\]
is dense in $L^{2}\left(\partial B_{k},\mu\right)$. 
\end{cor}

\begin{proof}
By \thmref{A7}, it is enough to prove that $Ker\left(T_{B}^{*}\right)=0$
where $T_{B}^{*}:L^{2}\left(\partial B_{k},\mu\right)\rightarrow\mathscr{H}\left(K\right)$
is the adjoint operator, i.e., 
\begin{equation}
\left(T_{B}^{*}\varphi\right)\left(z\right)=\int_{\partial B_{k}}\frac{1}{1-\left\langle z,b\right\rangle _{\mathbb{C}^{k}}}\varphi\left(b\right)d\mu\left(b\right),\quad\varphi\in L^{2}\left(\partial B_{k},\mu\right).
\end{equation}
But it follows from (\ref{eq:B6})-(\ref{eq:B8}) and \thmref{B2}
that: 
\begin{align*}
\left(T_{B}^{*}\varphi\right)\left(z\right) & =0,\quad\forall z\in B_{k}\\
 & \Updownarrow\\
\int_{\partial B_{k}}\varphi\left(b\right)\left\langle z,b\right\rangle ^{n}d\mu\left(b\right) & =0,\quad\forall z\in B_{k},\:\forall n\in\mathbb{N}_{0}.
\end{align*}

Now introduce a suitable ONB $\left\{ e_{j}\right\} $ in $\mathbb{C}^{k}$
s.t. $\left\langle e_{j},b\right\rangle _{\mathbb{C}^{k}}=b_{j}$.
Upon change of coordinates on $\partial B_{k}$, we get $\int\varphi\left(b_{1},\cdots,b_{k}\right)b_{j}^{n}db_{j}=0$,
$\forall n\in\mathbb{N}_{0}$, $1\leq j\leq k$. Hence $\varphi=0$
in $L^{2}\left(B_{k},\mu\right)$ as claimed. 
\end{proof}

\begin{cor}
\label{cor:D4}Let $K$ be the Drury-Arveson kernel on $B_{k}$; see
(\ref{eq:k18}). Let $\mu$ be the measure on $\partial B_{k}$ in
(\ref{eq:k10}), and let $K^{\left(\mu\right)}$ be the corresponding
p.d. kernel on $\partial B_{k}$; see (\ref{eq:k19}). Let $\mathscr{H}\left(K\right)$
and $\mathscr{H}(K^{\left(\mu\right)})$ be the corresponding RKHSs,
with respective inner products $\left\langle \cdot,\cdot\right\rangle _{\mathscr{H}\left(K\right)}$
and $\left\langle \cdot,\cdot\right\rangle _{\mathscr{H}(K^{\left(\mu\right)})}$. 

Then for every $f\in\mathscr{H}\left(K\right)$, and $r\in\left(0,1\right)$,
i.e., $0<r<1$, set 
\begin{equation}
f_{r}\left(b\right):=f\left(rb\right),\quad b\in\partial B_{k},\label{eq:B12}
\end{equation}
we conclude that $f_{r}\in\mathscr{H}(K^{\left(\mu\right)})$, and
\begin{equation}
\sup_{r<1}\left\Vert f_{r}\right\Vert _{\mathscr{H}(K^{\left(\mu\right)})}=\left\Vert f\right\Vert _{\mathscr{H}\left(K\right)}.\label{eq:B13}
\end{equation}
\end{cor}

\begin{proof}[Proof (sketch)]
Let $f\in\mathscr{H}\left(K\right)$, then, by (\ref{eq:k7}), we
have 
\[
f_{r}\left(b\right)=\sum_{n=0}^{\infty}r^{n}\left\langle \xi_{n},b^{\otimes n}\right\rangle ,\quad\forall b\in\partial B_{k}
\]
with 
\[
\sum_{n=0}^{\infty}\left\Vert \xi_{n}\right\Vert _{\mathscr{H}_{n}^{sym}}^{2}=\left\Vert f\right\Vert _{\mathscr{H}\left(K\right)}^{2}.
\]
Hence 
\[
\left\Vert f_{r}\right\Vert _{\mathscr{H}(K^{\left(\mu\right)})}^{2}=\sum_{n=0}^{\infty}r^{2n}\int_{\partial B_{k}}\left|\left\langle \xi_{n},b^{\otimes n}\right\rangle \right|^{2}d\mu\left(b\right),
\]
and the desired conclusion follows.
\end{proof}
\begin{defn}
\label{def:D5}Let $K$ be a p.d. kernel on a set $X$, and let $\left(B,\mathscr{F}_{B},\mu,k^{\left(B\right)}\right)$
be a boundary space in the sense of (\ref{eq:k5}) in \defref{FS}.
We say that it is a \emph{topological boundary} iff (Def) there is
a one-parameter family of maps $s_{r}:B\rightarrow X$, $0<r<1$,
such that 
\begin{equation}
\sup_{r<1}\left\Vert f\circ s_{r}\right\Vert _{\mathscr{H}\left(K^{\left(\mu\right)}\right)}=\left\Vert f\right\Vert _{\mathscr{H}\left(K\right)}\label{eq:tb1}
\end{equation}
holds for all $f\in\mathscr{H}\left(K\right)$. 
\end{defn}

Note that (\ref{eq:B13}) in \corref{D4} shows that the Drury-Arveson
kernel (see \exaref{D1}) has a natural topological boundary. 

\subsection{Iterated Function Systems and Boundaries}

In some of our earlier work (see, e.g., \cite{MR3958137,MR3796641,MR3318644,MR2811284,MR1655831,MR1389918,MR1297015,MR1165867}),
we considered an example of a p.d. kernel on $\mathbb{D}\times\mathbb{D}$
where $\mathbb{D}$ is the disk in one complex dimension. In this
case, the boundary measure $\mu$ has its support on a Cantor subset
of the boundary circle $\mathbb{T}$ to $\mathbb{D}$. This is the
case when this boundary measure $\mu$ is the \emph{Cantor measure}
corresponding to scaling by 4, and omitting 2 of 4 subintervals in
each of the iteration steps which determine $\mu$. In this case,
we have explicit formulas for the corresponding decompositions. In
particular, $K^{\left(\mu\right)}=K$. See details below.
\begin{example}
Consider measures $\mu$ on $\mathbb{T}\simeq\left[0,1\right]$ for
which there is a subset $\Lambda\subset\mathbb{N}_{0}$ such that
$\left\{ e_{\lambda}\left(\theta\right)\mathrel{;}\lambda\in\Lambda\right\} $
is orthogonal in $L^{2}\left(\mathbb{T},\mu\right)$. We say that
$\left(\mu,\Lambda\right)$ is a \emph{spectral pair}. 

Especially, if $\mu$ is the Cantor IFS (iterated function system)
determined by $\frac{x}{4}$ and $\frac{x+2}{4}$ (see \figref{C4}),
we may take $\Lambda$ to be the following 
\begin{align*}
\Lambda_{4} & =\left\{ 0,1,4,5,16,17,20,21,64,65,\cdots\right\} \\
 & =\left\{ \sum\nolimits _{0}^{finite}b_{i}4^{i}\mid b_{i}\in\left\{ 0,1\right\} \right\} .
\end{align*}

We recall that the $\nicefrac{1}{4}$-Cantor measure $\mu$ (see Figure
\ref{fig:C4}) is the unique probability Borel measure on $\mathbb{R}$
satisfying 
\[
\int\varphi\,d\mu=\frac{1}{2}\left(\int\varphi\left(\frac{x}{4}\right)d\mu\left(x\right)+\int\varphi\left(\frac{x+2}{4}\right)d\mu\left(x\right)\right)
\]
for all Borel functions $\varphi$. It is also determined by its Fourier
transform
\begin{align*}
\widehat{\mu}\left(\xi\right) & =\int_{\mathbb{R}}e^{i2\pi\xi\cdot x}d\mu\left(x\right)=\int_{\mathbb{R}}e_{\xi}\left(x\right)d\mu\left(x\right)\\
 & =e_{\xi}\left(\frac{1}{3}\right)\prod_{n=0}^{\infty}\cos\left(\frac{\pi\xi/2}{4^{n}}\right),\quad\xi\in\mathbb{R}.
\end{align*}

In this case we have 
\begin{align*}
K_{\Lambda_{4}}\left(z,w\right) & =\sum_{\lambda\in\Lambda_{4}}\overline{w}^{\lambda}z^{\lambda}=\prod_{k=0}^{\infty}\left(1+\left(\overline{w}z\right)^{4^{k}}\right),\\
K_{\Lambda_{4}}^{*}\left(z,\theta\right) & =\lim_{r\rightarrow1}K_{\Lambda_{4}}\left(z,re\left(\theta\right)\right).
\end{align*}
Setting 
\[
K^{\left(\mu\right)}\left(z_{1},z_{2}\right):=\int_{0}^{1}K_{\Lambda_{4}}^{*}\left(z_{1},\theta\right)\overline{K_{\Lambda_{4}}^{*}\left(z_{2},\theta\right)}d\mu\left(\theta\right),\quad\forall\left(z_{1},z_{2}\right)\in\mathbb{D}\times\mathbb{D};
\]
then, it follows that
\[
K^{\left(\mu\right)}\left(z_{1},z_{2}\right)=K_{\Lambda_{4}}\left(z_{1},z_{2}\right).
\]
\end{example}

\begin{figure}
\includegraphics[width=0.4\textwidth]{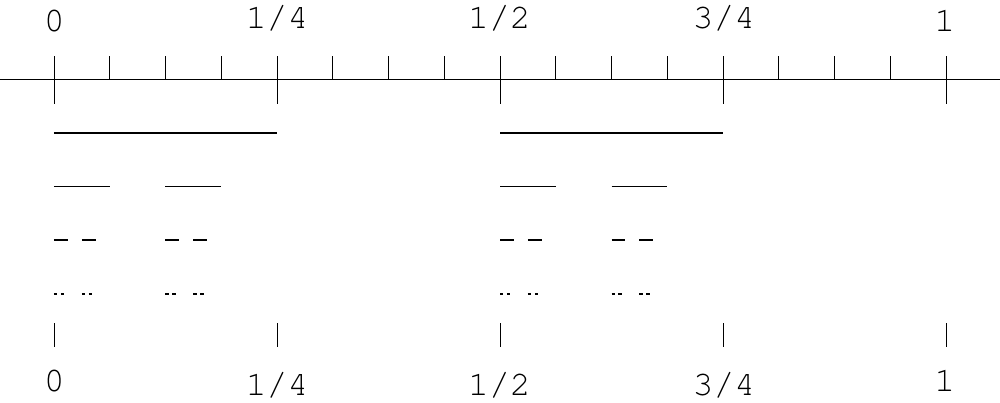}

\caption{\label{fig:C4}$\nicefrac{1}{4}$-Cantor set.}
\end{figure}

\section{\label{sec:gp}Gaussian Processes and Factorizations}

Here we discuss certain factorizations and boundary representations
which arise from a class of Gaussian processes and their generalized
Ito integrals.

In \secref{Intro} we considered positive definite (p.d.) functions
defined on $X\times X$ where $X$ was a general set. Below we study
certain p.d. functions which arise naturally when the set $X$ is
taken to be a prescribed sigma-algebra. In this setting, we shall
give an explicit formula for the corresponding reproducing kernel
Hilbert spaces (RKHSs). This particular family of p.d. functions and
associated RKHSs serve as useful tools in the context of \emph{Gaussian
processes}. 

Let $\left(B,\mathscr{F}_{B},\mu\right)$ be a measure space with
$\mu$ sigma-finite, and let $W^{\left(\mu\right)}$ be the corresponding
generalized Wiener processes (real valued), i.e., given by the three
conditions

(i) $W_{A}^{\left(\mu\right)}$, $A\in\mathscr{F}_{B}$ is Gaussian;

(ii) $\mathbb{E}(W_{A}^{\left(\mu\right)})=0$; 

(iii) the covariance formula holds,
\begin{equation}
\mathbb{E}\left(W_{A_{1}}^{\left(\mu\right)}W_{A_{2}}^{\left(\mu\right)}\right)=\mu\left(A_{1}\cap A_{2}\right),\quad\forall A_{1},A_{2}\in\mathscr{F}_{B}.
\end{equation}
For $k\in L^{2}\left(B,\mu\right)$, we then define the Ito-integral
\begin{equation}
V_{k}:=\int_{B}k\left(b\right)dW_{b}^{\left(\mu\right)},
\end{equation}
and we have 
\begin{equation}
\mathbb{E}\left(\overline{V_{k_{1}}}V_{k_{2}}\right)=\int_{B}\overline{k_{1}\left(b\right)}k_{2}\left(b\right)d\mu\left(b\right).
\end{equation}

In particular, if $k_{x}^{\left(B\right)}\left(b\right)=k^{\left(B\right)}\left(x,b\right):X\times B\rightarrow\mathbb{C}$
is such that $k^{\left(B\right)}\left(x,\cdot\right)\in L^{2}\left(B,\mu\right)$,
$\forall x\in X$, then 
\begin{equation}
\mathbb{E}\left(\overline{V_{k_{x}^{\left(B\right)}}}V_{k_{y}^{\left(B\right)}}\right)=\int_{B}\overline{k^{\left(B\right)}\left(x,b\right)}k^{\left(B\right)}\left(y,b\right)d\mu\left(b\right)=K^{\left(\mu\right)}\left(x,y\right),
\end{equation}
which is a p.d. kernel with contractive factorization. 

In general, if $V$ is a Gaussian process on $X$, and if $\left(B,\mathscr{F}_{B},\mu\right)$
is given, it is of interest to decide whether there exists $f\in L^{2}\left(\mu\right)$
such that 
\begin{equation}
V=\int_{B}f\left(b\right)dW_{b}^{\left(\mu\right)}.\label{eq:k21}
\end{equation}
We may define $V\ll W^{\left(\mu\right)}$ when $V$ has such a representation
(\ref{eq:k21}).
\begin{thm}
Let $\left(B,\mathscr{F}_{B},\mu\right)$ be a sigma-finite measure
space, and on $\mathscr{F}_{B}$ define the following kernel 
\begin{equation}
K\left(A_{1},A_{2}\right):=\mu\left(A_{1}\cap A_{2}\right),\quad A_{i}\in\mathscr{F}_{B}.\label{eq:k17}
\end{equation}
Then the reproducing kernel Hilbert space (RKHS) $\mathscr{H}$ of
this p.d. kernel is as follows:

The elements $F$ in $\mathscr{H}$ are signed measures on $\mathscr{F}_{B}$
which satisfy: 
\begin{equation}
F\ll\mu,\label{eq:km1}
\end{equation}
absolute continuity, with Radon-Nikodym derivative
\begin{equation}
\frac{dF}{d\mu}\in L^{2}\left(B,\mu\right).
\end{equation}
For the RKHS norm from $\mathscr{H}$, we have 
\begin{equation}
\left\Vert F\right\Vert _{\mathscr{H}}^{2}=\int_{B}\left|\frac{dF}{d\mu}\right|^{2}d\mu.\label{eq:km2}
\end{equation}
\end{thm}

\begin{proof}
It is clear that the conditions (\ref{eq:km1})--(\ref{eq:km2})
define a Hilbert space. Since the space spanned by the indicator functions
\begin{equation}
\left\{ \chi_{A}\mid A\in\mathscr{F}_{B},\:\mu\left(A\right)<\infty\right\} 
\end{equation}
is dense in $L^{2}\left(B,\mu\right)$, we need only verify that conditions
(\ref{enu:rk1}) and (\ref{enu:rk2}) from (\ref{eq:A1}) are satisfied.
Condition (\ref{enu:rk1}) is clear since if $A\in\mathscr{F}_{B}$
and $\mu\left(A\right)<\infty$, then $F_{A}\left(\cdot\right)=\mu\left(A\cap\cdot\right)$
is in $\mathscr{H}$, and, for all $G\in\mathscr{H}$, we then have:
\[
\left\langle F_{A},G\right\rangle _{\mathscr{H}}=\int_{B}\chi_{A}\left(\frac{dG}{d\mu}\right)d\mu=\int_{A}\frac{dG}{d\mu}d\mu=G\left(A\right),
\]
which is the desired conclusion (\ref{enu:rk2}).
\end{proof}

\section{Applications to Machine Learning}

One of the more recent applications of kernels and the associated
reproducing kernel Hilbert spaces (RKHS) is to optimization, also
called kernel-optimization. See \cite{MR3803845,MR2933765}. In the
context of machine learning, it refers to training-data and feature
spaces. In the context of numerical analysis, a popular version of
the method is used to produce splines from sample points; and to create
best spline-fits. In statistics, there are analogous optimization
problems going by the names \textquotedblleft least-square fitting,\textquotedblright{}
and \textquotedblleft maximum-likelihood\textquotedblright{} estimation.
In the latter instance, the object to be determined is a suitable
probability distribution which makes \textquotedblleft most likely\textquotedblright{}
the occurrence of some data which arises from experiments, or from
testing. 

A major theme in machine learning is inference from finite samples
in high dimensional spaces. In general, there are two types of learning
problems: (i) supervised learning (e.g., classification, regression),
and (ii) un-supervised learning (e.g., clustering). We shall consider
(i) in this section.

What these methods have in common is a minimization (or a max problem)
involving a \textquotedblleft quadratic\textquotedblright{} expression
$Q$ with two terms. The first in $Q$ measures a suitable $L^{2}\left(\mu\right)$-square
applied to a difference of a measurement and a \textquotedblleft best
fit.\textquotedblright{} The latter will then to be chosen from anyone
of a number of suitable reproducing kernel Hilbert spaces (RKHS).
The choice of kernel and RKHS will serve to select desirable features.
So we will minimize a quantity $Q$ which is the sum of two terms
as follows: (i) a $L^{2}$-square applied to a difference, and (ii)
a penalty term which is a RKHS norm-squared. (See eq. (\ref{eq:ao2}).)
In the application to determination of splines, the penalty term may
be a suitable Sobolev normed-square; i.e., $L^{2}$ norm-squared applied
to a chosen number of derivatives. Hence non-differentiable choices
will be \textquotedblleft penalized.\textquotedblright{} 

In all of the cases, discussed above, there will be a good choice
of (i) and (ii), and we show that there is then an explicit formula
for the optimal solution; see eq (\ref{eq:ao5}) in \thmref{opF}
below.

Let $X$ be a set, and let $K:X\times X\longrightarrow\mathbb{C}$
be a positive definite (p.d.) kernel. Let $\mathscr{H}\left(K\right)$
be the corresponding reproducing kernel Hilbert space (RKHS). Let
$\mathscr{B}$ be a sigma-algebra of subsets of $X$, and let $\mu$
be a positive measure on the corresponding measure space $\left(X,\mathscr{B}\right)$.
We assume that $\mu$ is sigma-finite. We shall further assume that
the associated operator $T$ given by 
\begin{equation}
\mathscr{H}\left(K\right)\ni f\xrightarrow{\;T\;}\left(f\left(x\right)\right)_{x\in X}\in L^{2}\left(\mu\right)\label{eq:ao1}
\end{equation}
is densely defined and closable.

Fix $\beta>0$, and $\psi\in L^{2}\left(\mu\right)$, and set 
\begin{equation}
Q_{\psi,\beta}\left(f\right)=\left\Vert \psi-Tf\right\Vert _{L^{2}\left(\mu\right)}^{2}+\beta\left\Vert f\right\Vert _{\mathscr{H}\left(K\right)}^{2}\label{eq:ao2}
\end{equation}
defined for $f\in\mathscr{H}\left(K\right)$ , or in the dense subspace
$dom\left(T\right)$ where $T$ is the operator in (\ref{eq:ao1}).
Let 
\begin{equation}
L^{2}\left(\mu\right)\xrightarrow{\;T^{*}\;}\mathscr{H}\left(K\right)\label{eq:ao3}
\end{equation}
be the corresponding adjoint operator, i.e., 
\begin{equation}
\left\langle F,T^{*}\psi\right\rangle _{\mathscr{H}\left(K\right)}=\left\langle Tf,\psi\right\rangle _{L^{2}\left(\mu\right)}=\int_{X}\overline{f\left(s\right)}\psi\left(s\right)d\mu\left(s\right).\label{eq:ao4}
\end{equation}

\begin{thm}
\label{thm:opF}Let $K$, $\mu$, $\psi$, $\beta$ be as specified
above; then the optimization problem 
\[
\inf_{f\in\mathscr{H}\left(K\right)}Q_{\psi,\beta}\left(f\right)
\]
has a unique solution $F$ in $\mathscr{H}\left(K\right)$, it is
\begin{equation}
F=\left(\beta I+T^{*}T\right)^{-1}T^{*}\psi\label{eq:ao5}
\end{equation}
where the operators $T$ and $T^{*}$ are as specified in (\ref{eq:ao1})-(\ref{eq:ao4}).
\end{thm}

\begin{proof}
(Sketch) We fix $F$, and assign $f_{\varepsilon}:=F+\varepsilon h$
where $h$ varies in the dense domain $dom\left(T\right)$ from (\ref{eq:ao1}).
For the derivative $\frac{d}{d\varepsilon}\big|_{\varepsilon=0}$
we then have: 
\[
\frac{d}{d\varepsilon}\big|_{\varepsilon=0}Q_{\psi,\beta}\left(f_{\varepsilon}\right)=2\Re\left\langle h,\left(\beta I+T^{*}T\right)F-T^{*}\psi\right\rangle _{\mathscr{H}\left(K\right)}=0
\]
for all $h$ in a dense subspace in $\mathscr{H}\left(K\right)$.
The desired conclusion follows.
\end{proof}

\subsection{Manifold Learning}

A more recent application is the so-called \emph{manifold learning}.
The basic idea is that the input data often lives in a submanifold
(nonlinear) of a much higher dimensional space, and the question is
how to extract and encode such information. (See \figref{sr} as a
standard illustration.) This extends the classical principal component
analysis (PCA), which is widely used to extract linear subspaces.
There is a collection of algorithms in manifold learning, and most
of these are based on the ``\emph{kernel trick}'', which itself
dates back at least to \cite{MR2418654}; see also \cite{MR1864085,MR2186447,MR2327597,MR3450534,MR3560092}.
The kernel trick refers to mapping the input data $X$ into a higher
(usually infinite) dimensional Hilbert space $\mathscr{H}$, called
a \emph{feature space}, and then apply standard learning algorithms
there. \figref{svm} illustrates this point of view using examples
of kernel \emph{support vector machine} (K-SVM). It shows a binary
classification problem, where the two classes are not linearly separable
in the input space, but by using a Gaussian kernel, the two classes
are linearly separated in the associated Gaussian RKHS.

\begin{figure}[H]
\includegraphics[width=0.35\columnwidth]{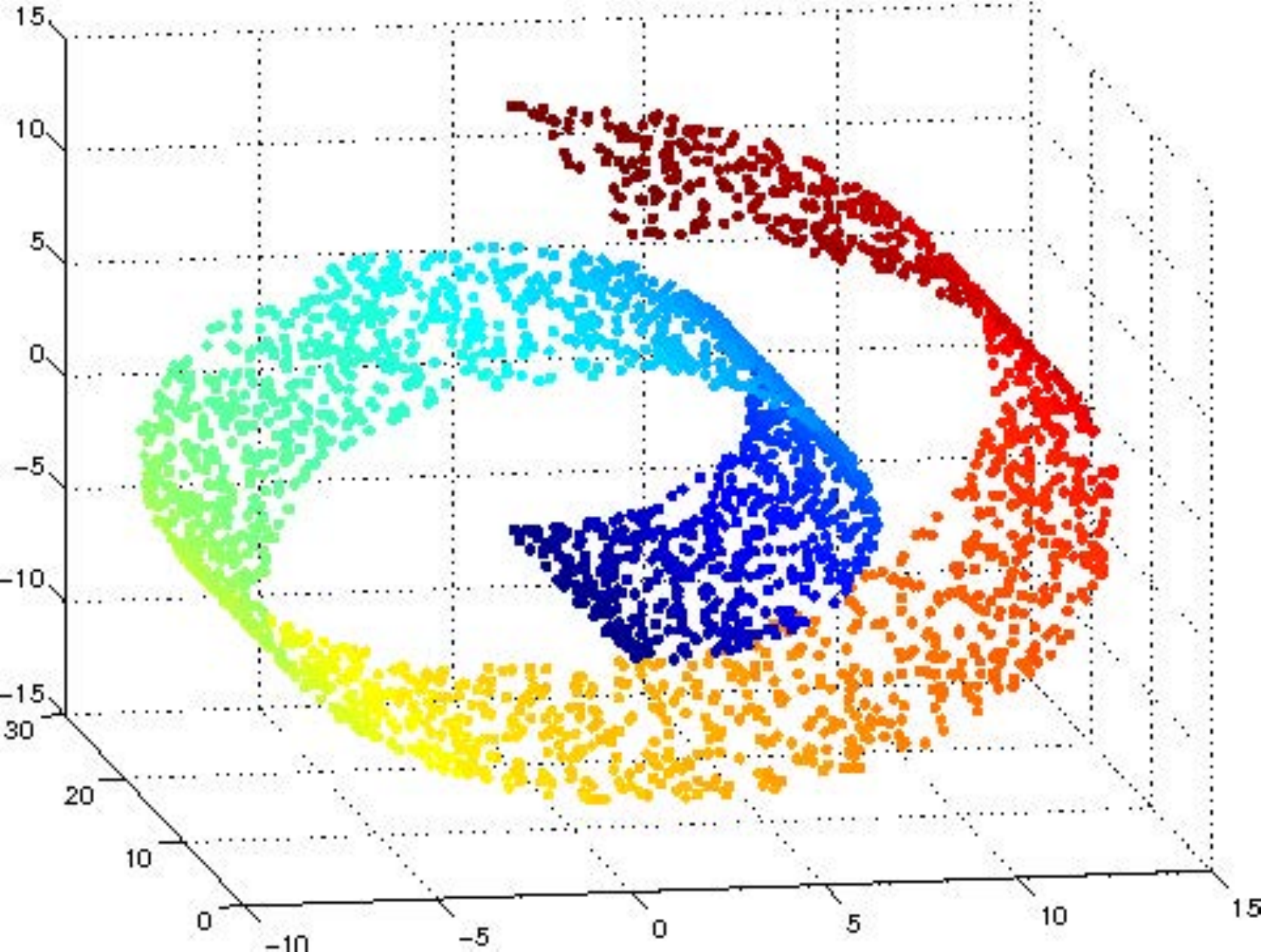}

\caption{\label{fig:sr}A Swiss role.}
\end{figure}

\begin{figure}[H]
\subfloat[Binary classification. No linear decision boundary.]{\includegraphics[width=0.45\columnwidth]{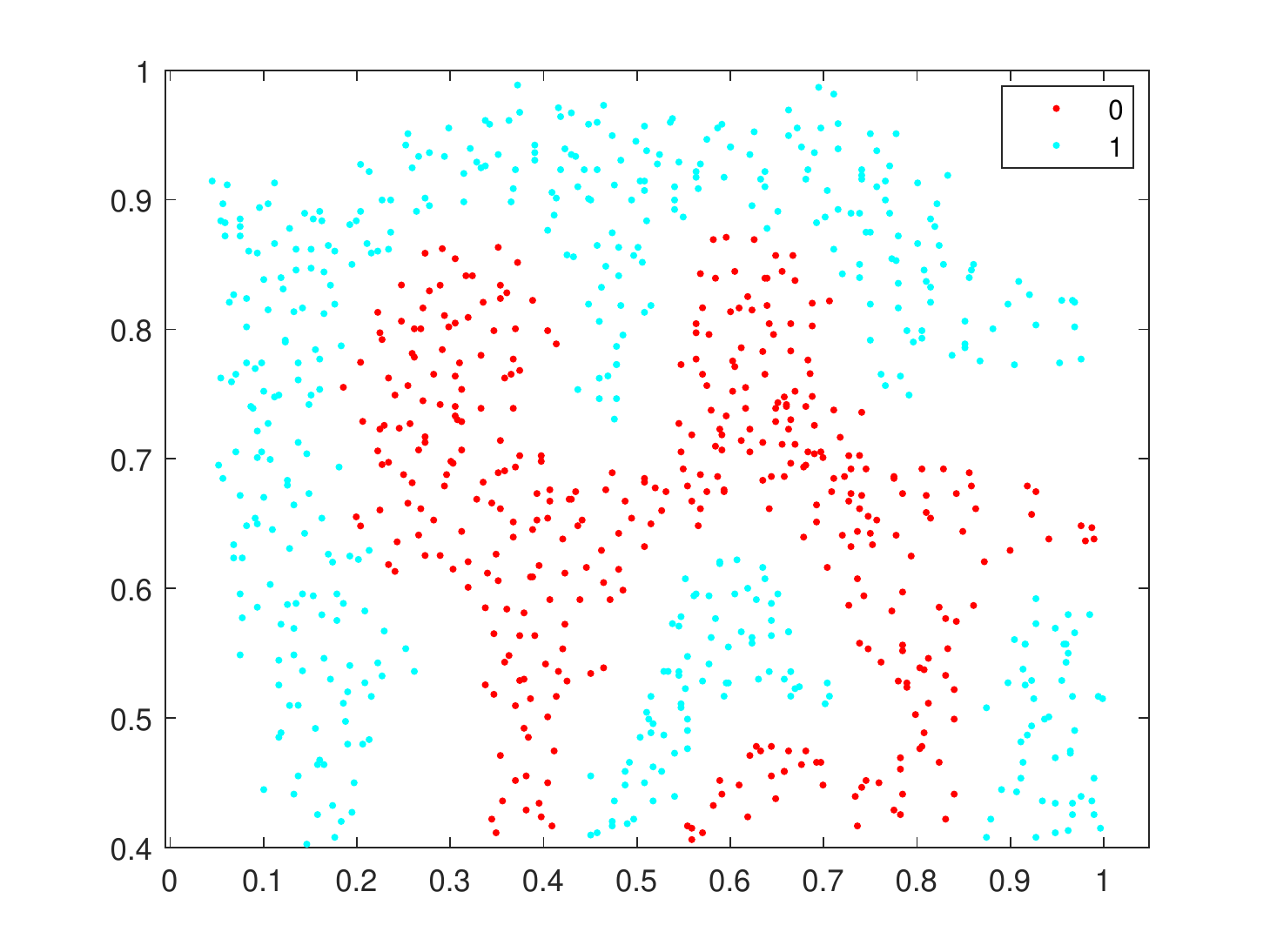}

}\hfill{}\subfloat[Linear boundary via Gaussian kernel.]{\includegraphics[width=0.45\columnwidth]{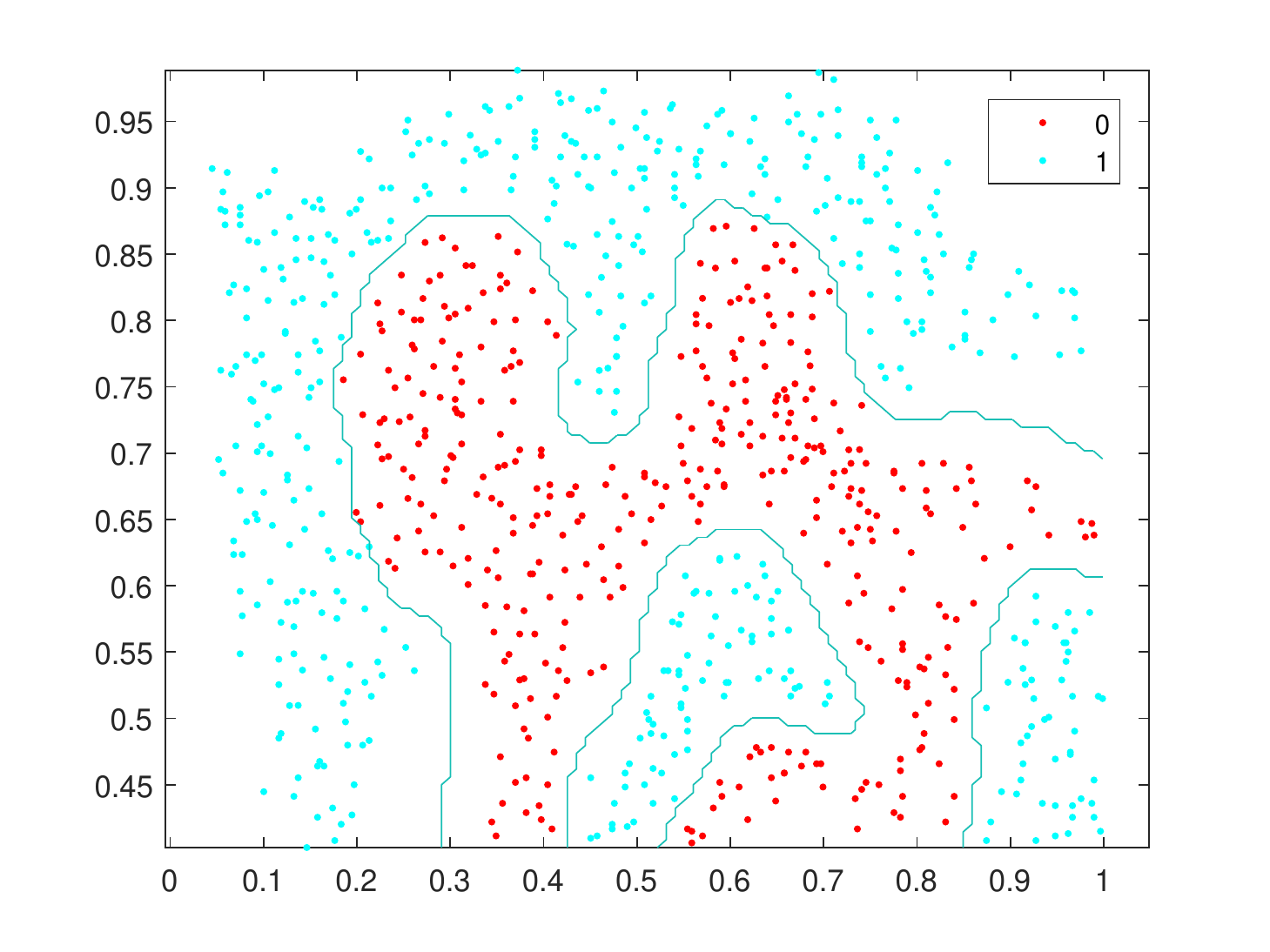}

}

\caption{\label{fig:svm}SVM using Gaussian kernel.}
\end{figure}

This idea sounds contra-intuitive as the higher dimensional spaces
are usually much more complicated; but the key is that the inner product
in $\mathscr{H}$ can be carried out via the kernel $K$. That is,
one chooses a \emph{feature map} $\pi:X\rightarrow\mathscr{H}$ such
that 
\begin{equation}
K\left(x_{i},x_{j}\right)=\left\langle \pi\left(x_{i}\right),\pi\left(x_{j}\right)\right\rangle _{\mathscr{H}}.
\end{equation}

In view of previous sections, we see the kernel trick amounts to certain
factorizations of p.d. kernels. It is well known that there are many
choices of the feature space $\mathscr{H}$. A canonical choice is
the RKHS $\mathscr{H}\left(K\right)$ of the p.d. kernel as in \secref{Intro},
or by general theory $K$ may be realized as the covariance of a Gaussian
process, see e.g., \secref{gp}. For the connections between these
two realizations, we have:
\begin{thm}[\cite{MR4020693}]
 The following are equivalent:
\begin{itemize}
\item Factorization of $K$:\textbf{ }There exists a measure space $\left(M,\mathscr{F},\mu\right)$,
and $\left\{ k_{x}\right\} _{x\in X}\subset L^{2}\left(M,\mu\right)$,
s.t. 
\[
K\left(x,y\right)=\int_{M}\overline{k_{x}\left(s\right)}k_{y}\left(s\right)d\mu\left(s\right).
\]
\item Disintegration of $V_{x}$: There exists $\left(M,\mathscr{F},\mu\right)$,
and $\left\{ k_{x}\right\} _{x\in X}\subset L^{2}\left(M,\mu\right)$,
s.t.
\[
V_{x}=\int_{\Omega}k_{x}\left(s\right)dW_{s}^{\left(\mu\right)}.
\]
Here, $\left\{ V_{x}\right\} _{x\in X}$ is the mean zero Gaussian
process, satisfying $\mathbb{E}(\overline{V}_{x}V_{y})=K\left(x,y\right)$.
\end{itemize}
\end{thm}

\begin{proof}
We only sketch the main ideas, and the reader is referred to \cite{MR4020693}
for more details.

If $K\left(x,y\right)=\int_{M}\overline{k_{x}\left(s\right)}k_{y}\left(s\right)d\mu\left(s\right)$,
set 
\[
V^{\left(\mu\right)}:=\int_{M}k_{x}\left(s\right)dW_{s}^{\left(\mu\right)}.
\]
Then $V^{\left(\mu\right)}\in L^{2}\left(\Omega,\mathbb{P}^{\left(\mu\right)}\right)$,
and $V\cong V^{\left(\mu\right)}$.

Conversely, assume $\left\{ V_{x}\right\} _{x\in X}$ admits a disintegration
$V_{x}=\int_{M}k_{x}\left(s\right)dW_{s}^{\left(\mu\right)}$. Then,
by Ito-isometry,
\[
\mathbb{E}\left[\overline{V_{x}}V_{y}\right]=\int_{M}\overline{k_{x}\left(s\right)}k_{y}\left(s\right)d\mu.
\]
\end{proof}

\subsection{Conclusions}

We emphasize here feature spaces that are $L^{2}$-spaces of sigma-finite
measures, and the factorizations as in the form of \defref{FS}.
\begin{acknowledgement*}
The co-authors thank the following colleagues for helpful and enlightening
discussions: Professors Daniel Alpay, Sergii Bezuglyi, Ilwoo Cho,
Myung-Sin Song, Wayne Polyzou, and members in the Math Physics seminar
at The University of Iowa. 
\end{acknowledgement*}
\bibliographystyle{amsalpha}
\bibliography{ref}

\end{document}